\documentclass[11pt, a4paper]{article}
\usepackage{amsfonts}
\usepackage{amsmath}

\usepackage{amsthm,amsfonts,color}
\usepackage{amssymb}
\usepackage{graphicx}
\usepackage{setspace}

\newtheorem{theorem}{Theorem}[section]
\newtheorem{corollary}{Corollary}[section]
\newtheorem{conjecture}{Conjecture}
\newtheorem{lemma}{Lemma}[section]

\newtheorem{claim}{Claim}
\newtheorem{definition}{Definition}
\newtheorem{proposition}{Proposition}[section]

\newcommand{\rem}{{\rm rem}}

\usepackage{geometry}
\begin{document}
\begin{spacing}{1.2}

\title{Uniform sets in a family with restricted intersections
\thanks{The first author is supported by NSFC (No. 11601430) and China Postdoctoral Science Foundation (No. 2016M590969);
the second author is supported by  NSFC (No. 11601429);
and the fourth author is supported by NSFC (Nos. 11571135 and 11671320).}}
\author{\quad Yandong Bai $^{a,b}$,
\quad Binlong Li $^{a,b}$,
\quad Jiuqiang Liu $^{b,c,}$\thanks{Corresponding author.
E-mail addresses:
bai@nwpu.edu.cn (Y. Bai),
binlongli@nwpu.edu.cn (B. Li),
jliu@emich.edu (J. Liu),
sgzhang@nwpu.edu.cn (S. Zhang).
},
\quad Shenggui Zhang $^{a,b}$\\[2mm]
\small $^{a}$ Department of Applied Mathematics, Northwestern Polytechnical University, \\
\small Xi'an, Shaanxi 710129, China\\
\small $^{b}$ Xi'an-Budapest Joint Research Center for Combinatorics, Northwestern Polytechnical University, \\
\small Xi'an, Shaanxi 710129, China\\
\small $^{c}$ Department of Mathematics, Eastern Michigan University, \\
\small Ypsilanti, MI 48197, USA
}
\date{\today}
\maketitle

\begin{abstract}
Let $\mathcal{F}$ be a family of subsets of $[n]=\{1,\ldots,n\}$ and
let $L$ be a set of nonnegative integers.
The family $\mathcal{F}$ is \emph{$L$-intersecting} if
$|F\cap F'|\in L$ for every two distinct members $F,F'\in\mathcal{F}$;
and $\mathcal{F}$ is $k$-uniform if all its members have the same size $k$.
A large variety of problems and results in extremal set theory
concern on $k$-uniform $L$-intersecting families. Many attentions
are paid to finding the maximum size of a family among all
$k$-uniform $L$-intersecting families with prescribed $n,k$ and $L$.
In this paper, from another point of view, we propose and
investigate the problem of estimating the maximum size of a member
in a family among all uniform $L$-intersecting families with size
$m$, here $n,m$ and $L$ are prescribed. Our results aim to find out
more precise relations of $n,m,k$ and $L$.

\medskip
\noindent
{\bf Keywords:} uniform intersecting family; Fisher's inequality; Erd\H{o}s-Ko-Rado theorem; extremal set theory
\smallskip
\end{abstract}

\section{Introduction}\label{section: introduction}

For a positive integer $n$, we set $[n]=\{1,\ldots,n\}$. A family
$\mathcal{F}$ of subsets of $[n]$ is an $m$-\emph{family} if
$\mathcal{F}$ has $m$ members; $\mathcal{F}$ is $k$-\emph{uniform}
if every member of $\mathcal{F}$ has size $k$; and for a set $L$ of
nonnegative integers, $\mathcal{F}$ is $L$-\emph{intersecting} if
$|F\cap F'|\in L$ for every two distinct members $F$ and $F'$ in
$\mathcal{F}$. Specially, if $L$ consists of all positive integers
(i.e., $L=\mathbb{N}^*$), then an $L$-intersecting family is also
called an \emph{intersecting} family. We say $\mathcal{F}$ is
\emph{uniform} if $\mathcal{F}$ is $k$-uniform for some $k$.

Extremal set theory studies various types of intersecting families,
see, e.g.,
\cite{AHS1972,BF1980,BFKS2001,CL2009,DEGKM1997,DEF1978,DF1983,FS2004,Isbell1959,Majumdar1953,RCW1975,Ryser1973,Snevily1994,Snevily1999,Snevily2003,Talbot2004}.
Among them, a large variety of problems and results concern on
$k$-uniform $L$-intersecting families. Many attentions are paid to
finding the maximum size of a family among all $k$-uniform
$L$-intersecting families with prescribed $n,k$ and $L$, see, e.g.,
\cite{Bose1949,EKR1961,Fisher1940,Frankl1984,RCW1975,RT2006,Snevily1994,Snevily1999}.
The first important result of this type is Fisher's inequality.

\begin{theorem}[Fisher's inequality, see
\cite{Bose1949,Fisher1940}]\label{ThFisher} Let $\mathcal{F}$ be a
$k$-uniform $\{l\}$-intersecting $m$-family of distinct subsets of
$[n]$, where $l\geqslant 1$ is an integer. Then $m\leqslant n$.
\end{theorem}

The intersection set $L$ in the above theorem consists of one positive integer.
For $L$ consists of more than one integer,
in 1961, Erd\H{o}s, Ko and Rado \cite{EKR1961} proved the following classical result,
which is now famous as EKR theorem and has a remarkable number of generalizations and analogues during the last half century,
see, e.g., \cite{ABS1991,DF1983,Frankl1984,FF1983,FT2016,FW1981,GS2002,HZ2017,LY2014,LZLZ2016,Mubayi2007,Mubayi2006,Pyber1986,Snevily2003,Talbot2004}.

\begin{theorem}[Erd\H{o}s, Ko and Rado \cite{EKR1961}]\label{EKR theorem}
Let $\mathcal{F}$ be a $k$-uniform intersecting $m$-family of distinct subsets of $[n]$ with $1\leqslant k\leqslant n/2$.
Then
$$
m\leqslant \binom{n-1}{k-1},
$$
and for $1\leqslant k<n/2$ the equality holds only if $\mathcal{F}$ consists of all $k$-subsets with a common element.
\end{theorem}

For intersection sets consisting of $s$ general nonnegative/positive integers,
the following two results have been proved.

\begin{theorem}[Ray-Chaudhuri and Wilson \cite{RCW1975}]\label{RCW theorem}
Let $\mathcal{F}$ be a $k$-uniform $L$-intersecting $m$-family of
subsets of $[n]$, where $L=\{l_{1},\ldots,l_{s}\}$ is a set of $s$
nonnegative integers. If $\max\{l_{1},\ldots,l_{s}\}\leqslant k-1$,
then
$$
m\leqslant \binom{n}{s}.
$$
\end{theorem}
\begin{theorem}[Heged\H{u}s \cite{Hegedus2015}]
Let $\mathcal{F}$ be a $k$-uniform $L$-intersecting $m$-family of
subsets of $[n]$, where $L=\{l_{1},\ldots,l_{s}\}$ is a set of $s$
positive integers with $l_{1}<l_{2}<\cdots <l_{s}$. If $n\geqslant
\binom{k^{2}}{l_{1}+1}s+l_{1}$, then
$$
m\leqslant \binom{n-l_{1}}{s}.
$$
\end{theorem}

In the above results on $k$-uniform $L$-intersecting $m$-families of
subsets of $[n]$, the authors fix $n,k,L$ and then consider how
large the size $m$ of a family could be. In this paper we
investigate this type problem in another direction. We make attempt
to estimate what is the maximum size of a member in a family among
all uniform $L$-intersecting $m$-families of subsets of $[n]$ with
prescribed $n,m$ and $L$. For a better presentation, we assume in
the following that the nonnegative integers in the considered
intersection set $L$ satisfy $l_{1}<l_{2}<\cdots <l_{s}$. Since
every two distinct members in $\mathcal{F}$ has less than $n$ common elements,
we will also assume that $l_{s}<n$. Define
\begin{equation}
\textrm{$\kappa_{_{L}}(n,m)=\max\{k$: there exits a $k$-uniform $L$-intersecting $m$-family of subsets of $[n]\}$},
\end{equation}
\begin{equation}
\textrm{$\mu_{_{L}}(n,k)=\max\{m$: there exits a $k$-uniform $L$-intersecting $m$-family of subsets of $[n]\}$}.
\end{equation}
We need to remark that in the definitions of $\kappa_{_{L}}(n,m)$ and
$\mu_{_{L}}(n,k)$ the subsets in the family are required to be distinct.
If there exists no $k$-uniform $L$-intersecting $m$-family of subsets of $[n]$ for any $k$,
then we define $\kappa_{_{L}}(n,m)=-\infty$.
Note that if $n\geqslant l_{s}+m$,
then $\kappa_{_{L}}(n,m)\geqslant l_{s}+1$,
as we can construct an $l_s$-uniform $L$-intersecting $m$-family $\{F_{1},\ldots,F_{m}\}$
in which $F_{i}=\{i\}\cup \{m+1,\ldots,m+l_{s}\}$ for each $1\leqslant i\leqslant m$.
One can also see from the above definitions that
$$
\kappa_{_L}(n,m)=\max\{k: \mu_{_L}(n,k)\geqslant m\}.
$$

Alternatively, we can restate our problem as follows.
Let $\mathcal{H}$ be a uniform hypergraph with $n$ vertices and $m$
hyperedges such that the intersection of every two hyperedges has size in $L$.
For given $n,m$ and $L$,
we want to know what is the maximum size of a hyperedge among all uniform hypergraphs
satisfying the above conditions.

We now present an extension concept of the $L$-intersecting families.
For an integer $t\geqslant 2$, a family $\mathcal{F}$ is
\emph{$t$-wise $L$-intersecting} if the intersection of every $t$
members in $\mathcal{F}$ has size in $L$. So an
$L$-intersecting family is 2-wise $L$-intersecting. We define
\begin{equation}
\textrm{$\kappa_{_{L}}^{t}(n,m)=\max\{k$: there exits a $k$-uniform
$t$-wise $L$-intersecting $m$-family of subsets of $[n]\}$}.
\end{equation}

The rest of this paper is organized as follows.
In the next section we study (2-wise) $L$-intersections for $L$ consisting of one integer.
We obtain exact values of $\kappa_{_{\{l\}}}(n,m)$ for
$1\leqslant m\leqslant 4$, and afterwards, we present both a lower
bound and an upper bound of $\kappa_{_{\{l\}}}(n,m)$ for general
$m$. In Section \ref{section: general intersection set} we consider
$L$-intersections for $L=\{0,1,\ldots,l\}$. In particular, we show
that
$$
\lim_{n\rightarrow \infty}\frac{\kappa_{_{\{0,1\}}}(n,n)}{\sqrt{n}}=1.
$$
In Section \ref{section: t-wise} we consider $t$-wise
$L$-intersections for general $t\geqslant 2$ and
$L=\{l_{1},\ldots,l_{s}\}$, we obtain an exact value of
$\kappa_{_L}^{t}(n,m)$ for large $n$. Section \ref{section: proofs
for singleton L} is devoted to the proofs of the main results in
Sections \ref{section: singleton} and \ref{section: general
intersection set}. In Section \ref{section: concluding} we propose
a problem for further research.

\section{$L$-intersecting families with $L=\{l\}$}\label{section: singleton}

In this section we deal with the case that $L$ is a singleton $\{l\}$,
where $l\geqslant 0$ and $n\geqslant l+m$.
For convenience,
we will write $\kappa_{_l}(n,m)$ for $\kappa_{_{\{l\}}}(n,m)$ in the following.
It is easy to check that
$$
\kappa_{_0}(n,m)=\left\lfloor\frac{n}{m}\right\rfloor \mbox{ for all } m\geqslant 1 \mbox{ and } n\geqslant m.
$$
So from now on we assume that $l\geqslant 1$.

For the first case $m=1$,
it is not difficult to see that the ground set $[n]$ forms a singleton family of maximum member size.

\begin{proposition}
$\kappa_{_{l}}(n,1)=n$.
\end{proposition}

We will further obtain exact values of $\kappa_{_{l}}(n,m)$ for
$2\leqslant m\leqslant 4$. Despite that the first two results
$\kappa_{_{l}}(n,1)=n$ and
$\kappa_{_l}(n,2)=\left\lfloor(n+l)/2\right\rfloor$ are not
difficult to verify, the proofs for the cases $m=3$ and $m=4$ are
somehow complicated.

\begin{theorem}\label{thm: m=2}
$\kappa_{_l}(n,2)=\left\lfloor\dfrac{n+l}{2}\right\rfloor$ for all $n\geqslant
l+2$.
\end{theorem}

\begin{theorem}\label{thm: m=3}
$\kappa_{_l}(n,3)=\left\{\begin{array}{ll}
\lfloor(n+l)/2\rfloor, & l+3\leqslant n<3l;\\
l+\lfloor n/3\rfloor, & n\geqslant \max\{3l,l+3\}.
\end{array}\right.$
\end{theorem}

\begin{theorem}\label{thm: m=4}
$\kappa_{_l}(n,4)=\left\{\begin{array}{ll}
\lfloor(n+l)/2\rfloor, & l+4\leqslant n<2l;\\
\lfloor(3n+6l)/8\rfloor, & \max\{2l,l+4\}\leqslant n<6l;\\
\lfloor(n+6l)/4\rfloor,   & n\geqslant 6l.
\end{array}\right.$
\end{theorem}

For general $m$, we obtain a lower bound and an upper bound for
$\kappa_{_l}(n,m)$. In the following inequality, we assume
$\kappa_{_l}(n,m)=-\infty$ when $l$ is negative.

\begin{theorem}\label{thm: a lower bound}
$\kappa_{_l}(n,m)\geqslant\max\limits_{2\leqslant i\leqslant m}\left\{\kappa_{l-{m-2\choose
i-2}}\left(n-\dbinom{m}{i},m\right)+\dbinom{m-1}{i-1}\right\}$, for
all $m\geqslant l\geqslant 1$.
\end{theorem}

For the upper bound, we need some new necessary definitions and
notations. Here we suppose that $n,k,l$ are real numbers. Let
$X=[0,n]$ be the real interval and let $\lambda$ be the Lebesgue
measure on $X$. If $\mathcal{F}=\{F_1,\ldots,F_m\}$ is a family of
subsets of $X$ such that \\
(i) $\lambda(F_i)=k$ for all $i\in[m]$, and \\
(ii) $\lambda(F_i\cap F_j)=l$ for all distinct $i,j\in[m]$, \\
then we call $\mathcal{F}$ a {\em fractional $\{l\}$-intersecting
$k$-uniform $m$-family} of $X$. For given real numbers $n,l$ and
integer $m$, let $\kappa_{_l}^{frac}(n,m)$ be the largest real
number $k$ such that there exists a fractional $\{l\}$-intersecting
$k$-uniform $m$-family of $[0,n]$. We obtain the following result on
$\kappa_{_l}^{frac}(n,m)$, which may be of independent interest.

\begin{theorem}\label{thm: real number upper bound}
Let $n>l\geqslant 0$ be two real numbers and let $m\geqslant 1$ be an integer.
Then
$$
\kappa_{_l}^{frac}(n,m)={m-1\choose s-1}\alpha+{m-1\choose t-1}\beta,
$$
where
$$
s=\left\lfloor\frac{1+\sqrt{1+4m(m-1)l/n}}{2}\right\rfloor,~~
t=\left\lceil\frac{1+\sqrt{1+4m(m-1)l/n}}{2}\right\rceil,
$$ and
$(\alpha,\beta)$ is the solution of
\begin{equation}\label{equation: a and b}
\left\{\begin{array}{l}
\dbinom{m}{s}\alpha+\dbinom{m}{t}\beta=n;\\
\dbinom{m-2}{s-2}\alpha+\dbinom{m-2}{t-2}\beta=l.
\end{array}\right.
\end{equation}
\end{theorem}

It is not difficult to verify that $\kappa_{_l}^{frac}(n,m)$ is an upper bound of $\kappa_{_l}(n,m)$.

\begin{theorem}\label{thm: upper bound for general l}
$\kappa_{_l}(n,m)\leqslant \kappa_{_l}^{frac}(n,m)$ for all integers $n>l\geqslant 0$.
\end{theorem}

Here we remark that if the solution $(\alpha,\beta)$ of Equation
(\ref{equation: a and b}) consists of two integers then the equality
in Theorem \ref{thm: upper bound for general l} holds. It is not
difficult to see that for any given integers $l,m$, there are
infinitely many integers $n$ such that the solutions of Equation
(\ref{equation: a and b}) are integers, it therefore follows that
there are infinitely many examples showing the sharpness of the
upper bound in Theorem \ref{thm: upper bound for general l}.

\section{$L$-intersecting families with $L=\{0,1,\ldots,l\}$}  \label{section: general intersection set}

In this section we deal with the case $L=\{0,1,\ldots,l\}$, where
$l<n$ is a positive integer. For convenience, we will write
$\kappa_{_{\leqslant l}}(n,m)$ for $\kappa_{_{\{0,\ldots,l\}}}(n,m)$
in the following. We start with the following theorem by Deza,
Erd\H{o}s and Frankl.

\begin{theorem}[Deza, Erd\H{o}s and Frankl \cite{DEF1978}]\label{thm: DEF theorem}
Let $s\leqslant k\leqslant n$ be positive integers,
$L$ a set of $s$ nonnegative integers
and $\mathcal{F}$ an $L$-intersecting $k$-uniform family of subsets of $[n]$.
Then there exists $n_0=n_0(k,L)$ such that for $n>n_{0}$ we have
$$
m=|\mathcal{F}|\leqslant \prod_{i=1}^{s}\frac{n-l_{i}}{k-l_{i}}.
$$
\end{theorem}

For the special case $L=\{0,\ldots,l\}$, we have the following
result. We will give a simple proof for convenience.

\begin{theorem}\label{thm: an upper bound for at most l}
For fixed $n\geqslant k>l\geqslant 1$, we have $\mu_{_{\leqslant
l}}(n,k)\leqslant\dfrac{n(n-1)\cdots(n-l)}{k(k-1)\cdots(k-l)}$.
\end{theorem}

\begin{proof}
Suppose that $\mathcal{F}=\{F_1,\ldots,F_m\}$ is a
$\{0,\ldots,l\}$-intersecting $k$-uniform family of $[n]$. We will
show that $m\leqslant\frac{n(n-1)\cdots(n-l)}{k(k-1)\cdots(k-l)}$.
We use induction on $k$. If $k=l+1$, then clearly $m={n\choose k}$
and the assertion holds. So we assume that $k\geqslant l+2$.

For $x\in[n]$, we set $\mathcal{F}_x=\{F\in\mathcal{F}: x\in F\}$.
Clearly $\sum_{x\in[n]}|\mathcal{F}_x|=\sum_{i=1}^m|F_i|=mk$. Note
that $\mathcal{F}'_x=\{F\backslash\{x\}: F\in\mathcal{F}_x\}$ is a
$\{0,\ldots,l-1\}$-intersecting $(k-1)$-uniform family of
$[n]\backslash\{x\}$. By induction hypothesis,
$$|\mathcal{F}_x|=|\mathcal{F}'_x|\leqslant\frac{(n-1)\cdots(n-l)}{(k-1)\cdots(k-l)}.$$
It follows that
$m=\sum_{x\in[n]}|\mathcal{F}_x|/k\leqslant\frac{n(n-1)\cdots(n-l)}{k(k-1)\cdots(k-l)}$.
\end{proof}

The above theorem in fact gives an upper bound for $\kappa_{_{\leqslant l}}(n,m)$.
We will make use of the following lower bound for special $n,m$.

\begin{theorem}\label{ThPrime}
Let $l$ be a positive integer and let $p\geqslant l$ be a prime. Then
$\mu_{_{\leqslant l}}(p^{2},p)\geqslant p^{l+1}$ and
$\kappa_{_{\leqslant l}}(p^2,p^{l+1})\geqslant p$.
\end{theorem}

\begin{proof}
Let
$$
X=\{(x,y): x,y\mbox{ are integers with }0\leqslant x,y<p\}.
$$
We will find a $\{0,\ldots,l\}$-intersecting $p$-uniform
$p^{l+1}$-family of $X$. Set $\mathcal{F}=\{F_{a_0,\ldots,a_l}:
0\leqslant a_i<p, 0\leqslant i\leqslant l\}$, where
$$F_{a_0,\ldots,a_l}=\{(x,y)\in X:
y\equiv a_0+xa_1+x^2a_2+\cdots+x^la_l\pmod p\}.$$ We now show that
any two members in $\mathcal{F}$ have at most $l$ common elements.
Suppose that
$$(x_0,y_0),\ldots,(x_l,y_l)\in F_{a_0,\ldots,a_l}\cap F_{b_0,\ldots,b_l}.$$
Let
$$\upsilon=\left(\begin{array}{c}
  y_0\\
  y_1\\
  \vdots\\
  y_l
\end{array}\right),~X=\left(\begin{array}{cccc}
  1     & x_0   & \cdots    & x_0^l\\
  1     & x_1   & \cdots    & x_1^l\\
  \vdots    &\vdots &\ddots & \vdots\\
  1     & x_l   & \cdots    & x_l^l
\end{array}\right),~\alpha=\left(\begin{array}{c}
  a_0\\
  a_1\\
  \vdots\\
  a_l
\end{array}\right),~\beta=\left(\begin{array}{c}
  b_0\\
  b_1\\
  \vdots\\
  b_l
\end{array}\right).$$
Then
$$\upsilon\equiv X\alpha\pmod p,~\upsilon\equiv X\beta\pmod p,~\mbox{ and }X(\alpha-\beta)\equiv 0\pmod p.$$
Since $\alpha\neq\beta$, we have that $X$ is irreversible (in the
field $F_p$). That is
$$p|\det A=\prod_{0\leqslant i<j\leqslant l}(x_j-x_i).$$
Since $p$ is a prime, there exist $i,j$ such that $p|(x_j-x_i)$.
Since $0\leqslant x_i,x_j\leqslant p-1$,
we have $x_i=x_j$ and $y_i=y_j$.
\end{proof}

Now we deal with the case $L=\{0,1\}$ and $m=n$. It is worth noting
that $\kappa_{_{\leqslant 1}}(n,n)$ is a non-decreasing function.
Note also that any two distinct $(\lceil n/2\rceil+1)$-subsets of
$[n]$ have at least two common elements. So $\kappa_{_{\leqslant
1}}(n,n)\leqslant \lceil n/2\rceil$ is a trivial upper bound. But
this bound is far from being sharp. We shall show that
$\kappa_{_{\leqslant 1}}(n,n)=\Theta(\sqrt{n})$ in the following.
For any real number $x$, let $p(x)$ be the smallest prime which is not less than $x$.

\begin{theorem}\label{thm: sqrt n}
$ \left\lfloor\dfrac{n}{p(\sqrt{n})}\right\rfloor \leqslant
\kappa_{_{\leqslant 1}}(n,n)\leqslant
\sqrt{n-\dfrac{3}{4}}+\dfrac{1}{2};$ and thus
$\lim\limits_{n\rightarrow\infty}\dfrac{\kappa_{\leqslant
1}(n,n)}{\sqrt{n}}=1$.
\end{theorem}

The sharpness of the upper bound can be deduced from the result below,
and the lower bound can be reached when, e.g., $n$ is a square of a prime.

\begin{theorem}\label{thm: sharpness of an upper bound}
Let $q$ be a prime power.
Then
$$
\kappa_{_{\leqslant 1}}(n,n)=
\begin{cases}
q, & if ~n\in[q^{2},q^{2}+q];\\
q+1, & if~n=q^{2}+q+1.
\end{cases}
$$
\end{theorem}

\section{$t$-wise $L$-intersecting families}\label{section: t-wise}

This section is devoted to $t$-wise $L$-intersecting families with
general $t\geqslant 2$ and general intersection set $L=\{l_{1},\ldots,l_{s}\}$.
We first give a lower bound and an upper bound on
$\kappa_{_L}^{t}(n,m)$.

\begin{theorem}\label{thm: LowerUpper}
Let $L=\{l_{1},\ldots,l_{s}\}$ with $l_{s}>l_{s-1}>\cdots
>l_{1}\geqslant 0$, $n\geqslant l_{s}+\frac{m}{t-1}$ and $m\geqslant t$. Then
$$
\left\lfloor\frac{(n-l_s)(t-1)}{m}\right\rfloor+l_s\leqslant\kappa_{_L}^{t}(n,m)\leqslant\left\lfloor\frac{n(t-1)}{m}+\frac{l_s}{m}\binom{m}{t}\right\rfloor.
$$
\end{theorem}

\begin{proof}
Let $A$ be a subset of $[n]$ of size $l_s$, say
$A=\{n,n-1,\ldots,n-l_s+1\}$. Let $\mathcal{B}=\{B_1,\ldots,B_m\}$
be a uniform family of subsets of $[n-l_s]$ such that every element
of $[n-l_s]$ appears in at most $t-1$ members of $\mathcal{B}$. It
is not difficult to see that the family $\mathcal{B}$ exists with each $B_i$
has size
$$
\left\lfloor\frac{(n-l_s)(t-1)}{m}\right\rfloor\geqslant 1.
$$
Now let $\mathcal{F}=\{F_1,\ldots,F_m\}$ with
$$
F_i=A\cup B_i,~\mbox{for~each~} 1\leqslant i\leqslant m.
$$
It is not difficult to see that $\mathcal{F}$ is a $t$-wise $L$-intersecting
family of $k$-subsets of $[n]$ with
$$
k=\left\lfloor\frac{(n-l_s)(t-1)}{m}\right\rfloor+l_s.
$$
Thus
$\kappa_{_L}^t(n,m)\geqslant\left\lfloor\frac{(n-l_s)(t-1)}{m}\right\rfloor+l_s$.

Suppose that $\mathcal{F}$ is a $k$-uniform $t$-wise
$L$-intersecting $m$-family of subsets of $[n]$. We construct a
bipartite graph $G$ with bipartition sets $X=[n]$ and
$Y=\mathcal{F}$, such that for each $x\in X$ and $F\in Y$, $xF\in
E(G)$ if and only if $x\in F$. Then each vertex in $Y$ has exactly
$k$ neighbors in $X$ and the graph $G$ has $km$ edges. Note that
each $t$ vertices in $Y$ has at most $l_s$ common neighbors in $X$.
For any vertex $x\in X$ and any $t$ vertices $F_1,\ldots,F_t$ with
$xF_i\in E(G)$, $1\leqslant i\leqslant t$, we delete one of such $t$ edges, say
$xF_1$. One can check that this procedure will remove at most
$\binom{m}{t}l_s$ edges and it will yield a graph such that each
vertex in $X$ has at most $t-1$ neighbors in $Y$. Thus the number of
remaining edges is at most $n(t-1)$, that is
$$
km-\binom{m}{t}l_s\leqslant n(t-1).
$$
Thus
$k\leqslant\left\lfloor\frac{n(t-1)}{m}+\frac{l_s}{m}\binom{m}{t}\right\rfloor$.
\end{proof}

When $n$ is large enough, we can show that the upper bound in
Theorem \ref{thm: LowerUpper} is the exact value of
$\kappa_{_L}^{t}(n,m)$.

\begin{theorem}\label{thm: n is large enough}
Let $L=\{l_{1},\ldots,l_{s}\}$ with $l_{s}>l_{s-1}>\cdots
>l_{1}\geqslant 0$. If $m\geqslant t$ and $n\geqslant\binom{m}{t}l_{s}$, then
$$
\kappa_{_L}^{t}(n,m)=\left\lfloor\frac{n(t-1)}{m}+\frac{l_s}{m}\binom{m}{t}\right\rfloor.
$$
\end{theorem}

\begin{proof}
By Theorem \ref{thm: LowerUpper}, it suffices to show that
$\kappa_{_L}^{t}(n,m)\geqslant\left\lfloor\frac{n(t-1)}{m}+\frac{l_s}{m}\binom{m}{t}\right\rfloor$.
Let $\mathcal{A}=\{A_T: T\subseteq[m], |T|=t\}$ be a family of
pairwise disjoint $l_s$-subsets of $[n]$. The family $\mathcal{A}$
exists since $n\geqslant\binom{m}{t}l_{s}$. Let
$\mathcal{B}=\{B_1,\ldots,B_m\}$ be a uniform family of subsets of
$[n]\backslash(\bigcup\mathcal{A})$ such that every element of
$[n]\backslash(\bigcup\mathcal{A})$ appears in at most $t-1$ members
of $\mathcal{B}$. Then one can see that we can take $\mathcal{B}$
such that each $B_i$ has size
$$
\left\lfloor\frac{(n-\binom{m}{t}l_s)(t-1)}{m}\right\rfloor=\left\lfloor\frac{n(t-1)}{m}-\frac{(t-1)l_s}{t}\binom{m-1}{t-1}\right\rfloor.
$$
We construct a family $\mathcal{F}=\{F_1,\ldots,F_m\}$ by letting
$$
F_i=B_i\cup\bigcup_{i\in T}A_T.
$$
It is not difficult to see that $\mathcal{F}$ is an $L$-intersecting
family of $k$-subsets of $[n]$ with
$$
k=\binom{m-1}{t-1}l_s+\left\lfloor\frac{n(t-1)}{m}-\frac{(t-1)l_s}{t}\binom{m-1}{t-1}\right\rfloor
=\left\lfloor\frac{n(t-1)}{m}+\frac{l_s}{m}\binom{m}{t}\right\rfloor.
$$
Thus
$\kappa_{_L}^t(n,m)\geqslant\left\lfloor\frac{n(t-1)}{m}+\frac{l_s}{m}\binom{m}{t}\right\rfloor$.
\end{proof}

As a corollary of Theorem \ref{thm: n is large enough}, taking
$t=2$, we have the following result.

\begin{corollary}
If $n\geqslant \binom{m}{2}l$, then
$\kappa_{_l}(n,m)=\kappa_{_{\leqslant
l}}(n,m)=\left\lfloor\frac{n}{m}+\frac{(m-1)l}{2}\right\rfloor$.
\end{corollary}

\section{Proofs of some main theorems}\label{section: proofs for singleton L}

In this section we present the proofs of some theorems in Sections \ref{section: singleton} and \ref{section: general intersection set},
namely, Theorems \ref{thm: m=2}-\ref{thm: real number upper bound} in Section \ref{section: singleton},
and Theorems \ref{thm: sqrt n}, \ref{thm: sharpness of an upper bound} in Section \ref{section: general intersection set}.
In the following proof we do not require the members of the family $\mathcal{F}$ to be distinct.
Note that under the above assumption the value of $\kappa_{_L}(n,m)$ will not change when $n\geqslant l_{s}+m$.

Set $M=[m]=\{1,2,\ldots,m\}$
and let $\mathcal{F}=\{F_1,\ldots,F_m\}$ be an arbitrary $\{l\}$-intersecting $k$-uniform $m$-family of $[n]$.
We define a function
$$
\phi=\phi_{\mathcal{F}}: 2^M\rightarrow\mathbb{N}
$$
such that $\phi(A)$ is the number of elements in $[n]$
that contained in each $F_i$ with $i\in A$
but not in any $F_i$ with $i\notin A$, i.e.,
$$
\phi(A)=|\{a\in[n]: A=\{i: a\in F_i\}\}|, \mbox{ for all } A\subseteq M.
$$
By the definition of an $\{l\}$-intersecting uniform family of $[n]$,
we have the following equations.
\begin{equation} \label{EqConstraint}
\left\{\begin{split}
\sum_{A\subseteq M}\phi(A) &=n;\\
\sum_{x\in A}\phi(A) &=\sum_{y\in B}\phi(B), \mbox{ for all }x,y\in M;\\
\sum_{x,y\in A}\phi(A) &=l, \mbox{ for all }x,y\in M.
\end{split}\right.
\end{equation}
We call a function $\phi: 2^M\rightarrow\mathbb{N}$ satisfying
(\ref{EqConstraint}) an \emph{assignment} (or exactly, an
$(l,m,n)$-assignment), and the equivalent number $\sum_{x\in
A}\phi(A)$ for each $x\in M$ is the \emph{value} of $\phi$, denoted
by $v(\phi)$. So every $\{l\}$-intersecting uniform family of $[n]$
corresponds to an assignment. On the other hand, for every
$(l,m,n)$-assignment $\phi$, we can easily get an
$\{l\}$-intersecting uniform family $\mathcal{F}$ of $[n]$ such that
$\phi=\phi_{\mathcal{F}}$. So the problem to find largest size of
the subsets in an $\{l\}$-intersecting uniform families of $[n]$, is
transferred to maximize the value $v(\phi)$ among all
$(l,m,n)$-assignments.

For two assignments $\phi_1$ and $\phi_2$,
their difference $\tau=\phi_1-\phi_2$ satisfies the following equations.
\begin{equation} \label{EqExtender}
\left\{\begin{split}
\sum_{A\subseteq M}\tau(A) &=0;\\
\sum_{x\in A}\tau(A) &=\sum_{y\in A}\tau(A), \mbox{ for all }x,y\in M;\\
\sum_{x,y\in A}\tau(A) &=0, \mbox{ for all }x,y\in M.
\end{split}\right.
\end{equation}
We call a function $\tau: 2^M\rightarrow\mathbb{Z}$ satisfying (\ref{EqExtender}) an \emph{extender}.
Note that the \emph{value} of $\tau$ (the equivalent number $\sum_{x\in A}\tau(A)$) is $v(\tau)=v(\phi_1)-v(\phi_2)$.
An extender with value $i$ is called an \emph{$i$-extender},
and sometimes we call a 0-extender a \emph{regulator}.
Note that an extender image some subsets of $M$ to a negative number,
whereas an assignment has only nonnegative objects.

Let $\phi$ be an assignment and let $\tau$ be an extender.
If $\phi+\tau$ is also an assignment
(i.e., $\tau(A)<0$ implies $\phi(A)\geqslant-\tau(A)$ for all $A\subseteq M$),
then we say that $\tau$ is \emph{compatible} with $\phi$.

\begin{lemma}\label{LeExtender}
An assignment $\phi$ has maximal value if and only if there exists no
positive-extender $\tau$ that is compatible with $\phi$.
\end{lemma}

\begin{proof}
If there is another assignment $\phi'$ with $v(\phi')>v(\phi)$,
then $\tau=\phi'-\phi$ is a positive extender compatible with $\phi$.
If $\phi$ has a compatible positive extender $\tau$,
then $\phi'=\phi+\tau$ is an assignment with $v(\phi')>v(\phi)$.
\end{proof}

We use $\rem(n,m)$ to denote the remainder of $n$ divided by $m$.

\begin{proof}[\bf{Proof of Theorem \ref{thm: m=2}}]
One can check by (\ref{EqExtender}) that the extender $\tau_i$ in
the following table is the only $i$-extender (for $m=2$).

\begin{center}
{\footnotesize Table 1.}\\[1mm]
\begin{tabular}{c|c|c|c|c}
\hline\hline $A$ & $\emptyset$ & $1$ & $2$ & $M$\\
\hline $\tau_i$ & $-2i$ & $i$ & $i$ & $0$\\
\hline\hline
\end{tabular}
\end{center}
Let $\phi$ be an assignment such that $\phi(\emptyset)=\rem(n+l,2)$,
$\phi(\{1\})=\phi(\{2\})=\lfloor(n-l)/2\rfloor$ and $\phi(\{1,2\})=l$.
It follows that there exists no positive-extender compatible with $\phi$.
By Lemma \ref{LeExtender},
$\kappa_l(n,2)=v(\phi)=\lfloor(n+l)/2\rfloor$.
\end{proof}

Let $\tau$ and $\tau'$ be two positive-extenders.
We write
$\tau'\preccurlyeq\tau$ if $\tau'(A)<0$ implies
$\tau'(A)\geqslant\tau(A)$ for all $A\subseteq M$. If there are no
other $\tau'$ with $\tau'\preccurlyeq\tau$,
then $\tau$ is a \emph{critical} extender.

\begin{lemma}\label{LeCritical}
An assignment $\phi$ has maximal value if and only if there exists no
critical positive-extender $\tau$ that is compatible with $\phi$.
\end{lemma}

\begin{proof}
Note that if $\tau'\preccurlyeq\tau$ and $\tau$ is compatible with
$\phi$, then $\tau'$ is also compatible with $\phi$. Also note that
if $\tau$ is not critical, then there is a critical extender $\tau'$
with $\tau'\preccurlyeq\tau$. The assertion now can be deduced by
Lemma \ref{LeExtender} immediately.
\end{proof}

\begin{proof}[\bf{Proof of Theorems \ref{thm: m=3}}]
We first show that the positive extenders in the following table are
the only critical extenders when $m=3$.
\begin{center}
{\footnotesize Table 2.}\\[1mm]
\begin{tabular}{c|c|c|c|c|c|c|c|c}
\hline\hline $A$ & $\emptyset$ & $1$ & $2$ & $3$ & $12$ & $13$ & $23$ & $M$\\
\hline $\tau_0$ & $-3$ & $1$ & $1$ & $1$ & $0$ & $0$ & $0$ & $0$\\
\hline $\tau_1$ & $-2$ & $0$ & $0$ & $0$ & $1$ & $1$ & $1$ & $-1$\\
\hline $\tau_2$ & $-1$ & $-1$ & $-1$ & $-1$ & $2$ & $2$ & $2$ & $-2$\\
\hline $\tau_3$ & $0$ & $-2$ & $-2$ & $-2$ & $3$ & $3$ & $3$ & $-3$\\
\hline\hline
\end{tabular}
\end{center}
Let $\tau$ be an arbitrary positive extender. One can compute by
(\ref{EqExtender}) that
\begin{align*}
& \tau(\{1,2\})=\tau(\{1,3\})=\tau(\{2,3\})=-\tau(M),\\
& \tau(\{1\})=\tau(\{2\})=\tau(\{3\})=\tau(M)+v(\tau), \mbox{ and}\\
& \tau(\emptyset)=-\tau(M)-3v(\tau).
\end{align*}
Since $\tau$ is positive, we have $v(\tau)\geqslant 1$. If
$\tau(M)\geqslant 1$, then $\tau(\emptyset)\leqslant-6$, implying
that $\tau_0\preccurlyeq\tau$. If $\tau(M)=-1$ and $v(\tau)\geqslant
2$, then $\tau(\emptyset)\leqslant-5$; if $\tau(M)=-2$ and
$v(\tau)\geqslant 2$, then $\tau(\emptyset)\leqslant-4$; if
$\tau(M)=-3$ and $v(\tau)\geqslant 2$, then
$\tau(\emptyset)\leqslant-3$, implying that
$\tau_0\preccurlyeq\tau$. Suppose now $\tau(M)\leqslant-4$. If
$v(\phi)\leqslant-\tau(M)-2$, then
$\tau(\{1\})=\tau(\{2\})=\tau(\{3\})\leqslant-2$,
implying that $\tau_3\preccurlyeq\tau$.
If $v(\phi)\geqslant-\tau(M)-1$,
then $\tau(\emptyset)\leqslant 2\tau(M)+3\leqslant-5$,
implying that $\tau_0\preccurlyeq\tau$.
It follows that $\tau_0,\tau_1,\tau_2,\tau_3$ are the only critical extenders.

Now we prove the assertion. If $l\leqslant n<3l$, then let
$$
\phi(A)=\left\{\begin{array}{ll}
  \rem(n-l,2),      & |A|=0;\\
  0,    & |A|=1;\\
  \lfloor(n-l)/2\rfloor,    & |A|=2;\\
  \lceil(3l-n)/2\rceil,     & |A|=3.
\end{array}\right.
$$If $n\geqslant 3l$,
then let
$$
\phi(A)=\left\{\begin{array}{ll}
  \rem(n,3),    & |A|=0;\\
  \lfloor n/3\rfloor-l,   & |A|=1;\\
  l,    & |A|=2;\\
  0,    & |A|=3.
\end{array}\right.
$$

One can check that all $\tau_i$, $i=0,\ldots,3$, are not compatible
with $\phi$. By Lemma \ref{LeCritical}, $\phi$ has the maximum value, i.e., $\kappa_{_l}(n,3)=v(\phi)$.
We can therefore obtain the desired result.
\end{proof}

An assignment (or extender) $\phi$ is \emph{balanced} if $|A|=|B|$
implies $\phi(A)=\phi(B)$. Clearly if $m\leqslant 3$, then every
assignment is balanced. For $m\geqslant 4$, there will be unbalanced
assignments.

\begin{lemma}\label{LeBalanced}
Suppose that there is a balanced assignment with maximum value among
all assignments for a given $m$. An assignment $\phi$ has maximal
value if and only if there exists no balanced critical positive-extender
$\tau$ that is compatible with $\phi$.
\end{lemma}

\begin{proof}
Note that the difference of two balanced assignments is a balanced
extender. The assertion can be obtained similarly as the analysis of
Lemma \ref{LeExtender}.
\end{proof}

\begin{proof}[\bf{Proof of Theorem \ref{thm: m=4}}]
We first show that there is a balanced assignment for $m=4$.

\begin{claim}\label{ClBalancedM4}
There is a balanced assignment $\phi$ with maximum value among all
assignments.
\end{claim}

\begin{proof}
We will use the following regulators.
\begin{center}
{\footnotesize Table 3.}\\[1mm]
\begin{tabular}{c|c|c|c|c|c|c|c|c|c|c|c|c|c|c|c|c}
\hline\hline $A$ & $\emptyset$ & $1$ & $2$ & $3$ & $4$ & $12$ & $13$
& $14$ & $23$ & $24$ & $34$ & $123$ & $124$ & $134$ & $234$ & $M$\\
\hline $\rho_0$ & $1$ & $0$ & $-1$ & $-1$ & $-1$ & $0$ & $0$ & $0$ & $1$ & $1$ & $1$ & $0$ & $0$ & $0$ & $-1$ & $0$\\
\hline $\rho_1$ & $1$ & $1$ & $1$ & $-1$ & $-1$ & $-1$ & $0$ & $0$ & $0$ & $0$ & $1$ & $0$ & $0$ & $-1$ & $-1$ & $1$\\
\hline $\rho_2$ & $0$ & $0$ & $0$ & $0$ & $-1$ & $0$ & $0$ & $1$ & $0$ & $1$ & $1$ & $0$ & $-1$ & $-1$ & $-1$ & $1$\\
\hline\hline
\end{tabular}
\end{center}

Let $\phi$ be a maximum-value assignment such that
$$\varDelta_\phi=\max_{|A|=3}\phi(A)-\min_{|A|=3}\phi(A)$$ is as small as
possible. It is sufficient to show that $\varDelta_\phi=0$. Let
$x_i=\phi(M\backslash\{i\})$ for $i=1,\ldots,4$. One can compute by
(\ref{EqExtender}) that
$$
\phi(\{i\})=v(\phi)+\sum_{i=1}^4x_i-3l-2\phi(M)-x_i,~i=1,\ldots,4.
$$

If $x_1>\max\{x_2,x_3,x_4\}$, then $\phi(\{1\})<\phi(\{i\})$,
$i=2,3,4$, implying that $\phi(\{i\})\geqslant 1$, $i=2,3,4$. In
this case $\rho_0$ is compatible with $\phi$. It follows that
$\phi'=\phi+\rho_0$ has value $v(\phi')=v(\phi)$ and
$\varDelta_{\phi'}<\varDelta_\phi$, a contradiction. If
$x_1=x_2>\max\{x_3,x_4\}$, then $\phi(\{3\})=\phi(\{4\})\geqslant 1$
and $\phi\{1,2\}\geqslant 1$. In this case $\rho_1$ is compatible
with $\phi$. It follows that $\phi'=\phi+\rho_1$ has value
$v(\phi')=v(\phi)$ and $\varDelta_{\phi'}<\varDelta_\phi$, a
contradiction. If $x_1=x_2=x_3>x_4$, then $\phi(\{4\})\geqslant 1$.
In this case $\rho_2$ is compatible with $\phi$. It follows that
$\phi'=\phi+\rho_2$ has value $v(\phi')=v(\phi)$ and
$\varDelta_{\phi'}<\varDelta_\phi$, a contradiction. The other cases
are similarly. Thus we conclude that $\varDelta_\phi=0$. It follows
from (\ref{EqExtender}) that $\phi$ is balanced.
\end{proof}

Now we list the following extenders.
One can check that there exists no other balanced critical extender.
We omit the details here.
\begin{center}
{\footnotesize Table 4.}\\[1mm]
\begin{tabular}{c|c|c|c|c|c}
\hline\hline $A$ & $|A|=0$ & $|A|=1$ & $|A|=2$ & $|A|=3$ & $|A|=4$\\
\hline $\tau_0$ & $-2$ & $-1$ & $2$ & $-2$ & $2$\\
\hline $\tau_1$ & $0$ & $-2$ & $2$ & $-1$ & $0$\\
\hline $\tau_2$ & $-3$ & $0$ & $1$ & $-1$ & $1$\\
\hline $\tau_3$ & $2$ & $-3$ & $2$ & $0$ & $2$\\
\hline $\tau_4$ & $-1$ & $-1$ & $1$ & $0$ & $-1$\\
\hline $\tau_5$ & $-4$ & $1$ & $0$ & $0$ & $0$\\
\hline $\tau_6$ & $1$ & $-2$ & $1$ & $1$ & $-3$\\
\hline $\tau_7$ & $-2$ & $0$ & $0$ & $1$ & $-2$\\
\hline $\tau_8$ & $0$ & $-1$ & $0$ & $2$ & $-4$\\
\hline $\tau_{9}$ & $-1$ & $0$ & $-1$ & $3$ & $-5$\\
\hline $\tau_{10}$ & $0$ & $0$ & $-2$ & $5$ & $-8$\\
\hline\hline
\end{tabular}
\end{center}

Now we construct an assignment $\phi$ as follows.
If $l\leqslant n<2l$, then let
$$
\phi(A)=\left\{\begin{array}{ll}
  \rem(n-l,2),      & |A|=0;\\
  0,    & |A|=1;\\
  0,    & |A|=2;\\
  \lfloor(n-l)/2\rfloor,    & |A|=3;\\
  2l-n+\rem(n-l,2),     & |A|=4.
\end{array}\right.
$$
If $2l\leqslant n<6l-5$, then let $n-6l=-8q+r$, $0\leqslant r<8$, and
$$
\phi(A)=\left\{\begin{array}{ll}
  \rem(r,3),      & |A|=0;\\
  0,    & |A|=1;\\
  l-2q+\lfloor r/3\rfloor,    & |A|=2;\\
  q-\lfloor r/3\rfloor,    & |A|=3;\\
  \lfloor r/3\rfloor,     & |A|=4;
\end{array}\right.
$$
If $n=6l-5+r$, $0\leqslant r\leqslant 3$, then let
$$
\phi(A)=\left\{\begin{array}{ll}
  r,      & |A|=0;\\
  0,    & |A|=1;\\
  l-1,    & |A|=2;\\
  0,    & |A|=3;\\
  1,     & |A|=4.
\end{array}\right.
$$
If $n=6l-1$, then let
$$
\phi(A)=\left\{\begin{array}{ll}
  0,      & |A|=0;\\
  1,    & |A|=1;\\
  l-1,    & |A|=2;\\
  0,    & |A|=3;\\
  1,     & |A|=4.
\end{array}\right.
$$
If $n\geqslant 6l$, then let
$$
\phi(A)=\left\{\begin{array}{ll}
  \rem(n-6l,4),      & |A|=0;\\
  \lfloor(n-6l)/4\rfloor,    & |A|=1;\\
  l,    & |A|=2;\\
  0,    & |A|=3;\\
  0,     & |A|=4.
\end{array}\right.
$$

One can check that for each case,
any $\tau_i$, $i=0,\ldots,10$,
is not compatible with $\phi$.
By Claim \ref{ClBalancedM4} and Lemma \ref{LeBalanced},
$\phi$ has maximum value, i.e., $\kappa_{_l}(n,4)=v(\phi)$.
One can compute the desired result.
\end{proof}

\begin{proof}[\bf{Proof of Theorem \ref{thm: a lower bound}}]
Let $i$ be an arbitrary integer with $2\leqslant i\leqslant m$. Let
$\phi_i$ be an
$$
\left(l-{m-2\choose i-2},m,n-{m\choose i}\right)\mbox{-}assignment
$$
such that $v(\phi_i)$ is maximum, and let
$\pi_i:2^M\rightarrow\mathbb{N}$ be a function such that
$$\pi_i(A)=\left\{\begin{array}{ll}
1, & |A|=i;\\
0, & \mbox{ otherwise}.
\end{array}\right.$$
Then $\phi=\phi_i+\pi_i$ is an $(l,m,n)$-assignment with
$v(\phi)=v(\phi_i)+{m-1\choose i-1}$. Thus the assertion holds.
\end{proof}

\begin{proof}[\bf{Proof of Theorem \ref{thm: real number upper bound}}]
Let $\mathcal{F}=\{F_1,\ldots,F_m\}$ be a fractional $\{l\}$-intersecting uniform family of $X$.
As in the previous case, we define a function
$\phi_\mathcal{F}: 2^M\rightarrow \mathbb{R}^+\cup\{0\}$ such that
$$\phi(A)=\lambda(\{x\in X: \{i\in M: x\in F_i\}=A\}), \mbox{ for all
}A\subseteq M.$$ Thus $\phi=\phi_\mathcal{F}$ satisfies
\begin{equation} \label{EqReal}
\left\{\begin{split}
\sum_{A\subseteq M}\phi(A) &=n;\\
\sum_{x\in A}\phi(A) &=\sum_{y\in A}\phi(A), \mbox{ for all }x,y\in M;\\
\sum_{x,y\in A}\phi(A) &=l, \mbox{ for all }x,y\in M.
\end{split}\right.
\end{equation}
Now we will find the maximum value $v(\phi)$ among all assignments
satisfying (\ref{EqReal}). Recall that $\phi$ is balanced if
$|A|=|B|$ implies $\phi(A)=\phi(B)$ for all $A,B\subseteq M$.

\begin{claim}
There is a balanced assignment with maximum value.
\end{claim}

\begin{proof}
Let $\phi$ be an assignment with maximum value, and let $\varOmega$
be the symmetric group on $M$. For any $\sigma\in\varOmega$, we
define $\phi_\sigma$ as $$\phi_\sigma(A)=\phi(\sigma(A)), \mbox{ for
all }A\subseteq M.$$ Clearly $v(\phi_\sigma)=v(\phi)$, implying that
$\phi_\sigma$ has maximum value for all $\sigma\in\varOmega$. It
follows that
$$\phi^*=\frac{1}{m!}\sum_{\sigma\in\varOmega}\phi_\sigma$$
has value $v(\phi^*)=v(\phi)$, the maximum value as well.
It is not difficult to see that $\phi^*$ is balanced. This proves the claim.
\end{proof}

Now let $\phi$ be a balanced assignment with maximum value. For
convenience, we define
$$
\varphi: M^*=M\cup\{0\}\rightarrow\mathbb{R}^+\cup\{0\}
$$
such that $\varphi(i)=\phi(A)$ for all $A$ of size $i$. Therefore
\begin{equation} \label{EqVarphi}
\left\{\begin{split}
& \sum_{i=0}^m\dbinom{m}{i}\varphi(i)=n;\\
& \sum_{i=2}^m\dbinom{m-2}{i-2}\varphi(i)=l.
\end{split}\right.
\end{equation}

\begin{claim}\label{ClNeqzero}
There exists $s$, $0\leqslant s\leqslant m-1$ such that $\varphi(i)=0$ for all
$i\in M^*\backslash\{s,s+1\}$.
\end{claim}

\begin{proof}
We need the following fact.

\begin{lemma}\label{LeRst}
Suppose that $0\leqslant r<s<t\leqslant m$ are integers and $\alpha,\beta>0$
are real numbers. If
\begin{equation} \label{EqAlpha}
\left\{\begin{split}
& \dbinom{m}{r}\alpha+\dbinom{m}{t}\beta=\dbinom{m}{s},\\
& \dbinom{m-2}{r-2}\alpha+\dbinom{m-2}{t-2}\beta=\dbinom{m-2}{s-2},
\end{split}\right.
\end{equation}
then
$$\dbinom{m-1}{r-1}\alpha+\dbinom{m-1}{t-1}\beta<\dbinom{m-1}{s-1}.$$
\end{lemma}

\begin{proof}
If $r=0$, then
$$\dbinom{m-1}{t-1}\beta=\dbinom{m-1}{t-1}\dbinom{m-2}{s-2}/\dbinom{m-2}{t-2}=\dfrac{s-1}{t-1}\dbinom{m-1}{s-1}<\dbinom{m-1}{s-1}.$$
Thus we assume that $r\geqslant 1$. By (\ref{EqAlpha}), we can solve that
$$\alpha=\dfrac{\left|\begin{array}{cc}
  \dbinom{m}{s} & \dbinom{m}{t}\\
  \dbinom{m-2}{s-2} & \dbinom{m-2}{t-2}
\end{array}\right|}{\left|\begin{array}{cc}
  \dbinom{m}{r} & \dbinom{m}{t}\\
  \dbinom{m-2}{r-2} & \dbinom{m-2}{t-2}
\end{array}\right|}, \mbox{ and }\beta=\dfrac{\left|\begin{array}{cc}
  \dbinom{m}{r} & \dbinom{m}{s}\\
  \dbinom{m-2}{r-2} & \dbinom{m-2}{s-2}
\end{array}\right|}{\left|\begin{array}{cc}
  \dbinom{m}{r} & \dbinom{m}{t}\\
  \dbinom{m-2}{r-2} & \dbinom{m-2}{t-2}
\end{array}\right|}.$$
Thus the assertion is implied by

\begin{equation*}\begin{split}
  \dbinom{m-1}{r-1}\left|\begin{array}{cc}
  \dbinom{m}{s} & \dbinom{m}{t}\\
  \dbinom{m-2}{s-2} & \dbinom{m-2}{t-2}
\end{array}\right| & +\dbinom{m-1}{t-1}\left|\begin{array}{cc}
  \dbinom{m}{r} & \dbinom{m}{s}\\
  \dbinom{m-2}{r-2} & \dbinom{m-2}{s-2}
\end{array}\right|\\
& <\dbinom{m-1}{s-1}\left|\begin{array}{cc}
  \dbinom{m}{r} & \dbinom{m}{t}\\
  \dbinom{m-2}{r-2} & \dbinom{m-2}{t-2}
  \end{array}\right|.\end{split}
\end{equation*}
By taking a factor
$\dfrac{m}{m-1}\dbinom{m-1}{r-1}\dbinom{m-1}{s-1}\dbinom{m-1}{t-1}$,
we obtain
$$  \left|\begin{array}{cc}
  \dfrac{1}{s} & \dfrac{1}{t}\\
  s-1 & t-1
\end{array}\right|+\left|\begin{array}{cc}
  \dfrac{1}{r} & \dfrac{1}{s}\\
  r-1 & s-1
\end{array}\right|
  <\left|\begin{array}{cc}
  \dfrac{1}{r} & \dfrac{1}{t}\\
  r-1 & t-1
  \end{array}\right|.$$
That is
$$\frac{t-r}{s}<\frac{s-r}{t}+\frac{t-s}{r},$$ which can be checked directly.
\end{proof}

Now we prove the claim. Suppose that there are $0\leqslant r,t\leqslant m$
with $t\geqslant r+2$ such that $\varphi(r),\varphi(t)\neq 0$. Let $s$ be
an integer such that $r<s<t$. Let $\alpha,\beta$ be the solution of
(\ref{EqAlpha}). We take a coefficient $c$ such that
$c\alpha\leqslant\varphi(r)$ and $c\beta\leqslant\varphi(t)$. Now we let
$\varphi'$ be an assignment such that
$$\varphi'(i)=\left\{\begin{array}{ll}
  \varphi(i)-c\alpha,   & i=r;\\
  \varphi(i)+c,         & i=s;\\
  \varphi(i)-c\beta,    & i=t;\\
  \varphi(i),           & \mbox{otherwise}.
\end{array}\right.$$
By Lemma \ref{LeRst}, $\varphi'$ is an assignment with
$v(\varphi')>v(\varphi)$, a contradiction.
\end{proof}

\begin{claim}\label{ClGeqleq}
There exist $0\leqslant s,t\leqslant m$ with $\varphi(s),\varphi(t)>0$ and satisfying that
$$\dbinom{m-2}{s-2}/\dbinom{m}{s}\leqslant\frac{l}{n}, \mbox{ and }
\dbinom{m-2}{t-2}/\dbinom{m}{t}\geqslant\dfrac{l}{n}.$$ Moreover,
the first inequality is strict if and only if the second inequality
is strict.
\end{claim}

\begin{proof}
If for all $i$ with $\varphi(i)>0$,
either $\dbinom{m-2}{i-2}/\dbinom{m}{i}>\dfrac{l}{n}$;
or $\dbinom{m-2}{i-2}/\dbinom{m}{i}\geqslant\dfrac{l}{n}$ and
there exists an integer $t$ with $\dbinom{m-2}{t-2}/\dbinom{m}{t}>\dfrac{l}{n}$,
then by
(\ref{EqVarphi}),
$$\sum_{i=0}^m\dbinom{m-2}{i-2}\varphi(i)>\sum_{i=1}^m\dbinom{m}{i}\dfrac{l}{n}\varphi(i)=l,$$
a contradiction. The second assertion can be proved similarly.
\end{proof}

Now let $r$ be the positive solution of
$\dbinom{m-2}{r-2}/\dbinom{m}{r}=\dfrac{l}{n}$. By Claims
\ref{ClNeqzero} and \ref{ClGeqleq}, if $r$ is an integer,
then $s=t=r$; if $r$ is not an integer, then $s=\lfloor r\rfloor$
and $t=\lceil r\rceil$. Now the theorem can be deduced by
(\ref{EqVarphi}).
\end{proof}

\begin{proof}[\bf{Proof of Theorem \ref{thm: sqrt n}}]
We first show the limit part of the theorem.
It is not difficult to check that the limit of the upper bound is one.
For the limit of the lower bound,
one can see that it is a consequence of the following lemma,
which can be deduced from one result of Dusart in \cite{Du2018}.

\begin{lemma}
$\lim\limits_{x\rightarrow\infty}\dfrac{p(x)}{x}=1$.
\end{lemma}

Now we show the upper bound and the lower bound of
$\kappa_{_{\leqslant 1}}(n,n)$. For the upper bound, let
$\mathcal{F}$ be a $k$-uniform $\{0,1\}$-intersecting $n$-family of
subsets of $[n]$. By Theorem \ref{thm: an upper bound for at most
l},
$$n\leqslant \frac{n(n-1)}{k(k-1)}.$$ Thus we have
$k\leqslant \sqrt{n-3/4}+1/2$.

For the lower bound, we first show the following claim.

\begin{claim}\label{claim}
Let $p$ be a prime and let $t<p$ be a positive integer. Then
$\kappa_{_{\leqslant 1}}(p^2-tp,p^2-tp)\geqslant p-t$.
\end{claim}

\begin{proof}
Set $X=\{(x,y): 0\leqslant x,y<p\}$. From the proof of Theorem
\ref{ThPrime}, we can see that the family $\mathcal{F}=\{F_{a,b}:
0\leqslant a,b<p\}$ is a $p$-uniform $\{0,1\}$-intersecting $p^2$-family
of subsets of $X$, where
$$
F_{a,b}=\{(x,y): y\equiv a+bx\pmod p\}.
$$

Let $X'=\{(x,y): 0\leqslant x<p-t, 0\leqslant y<p\}$. So $X'$ is a
subset of $X$ of size $p^2-tp$. Let $F'_{a,b}=F_{a,b}\cap X'$ and
$\mathcal{F}'=\{F'_{a,b}: F_{a,b}\in\mathcal{F}\}$. Clearly
$\mathcal{F}'$ is a $(p-t)$-uniform $\{0,1\}$-intersecting family
with $p^2\geqslant p^2-tp$ members. This implies that
$\kappa_{_{\leqslant 1}}(p^2-tp,p^2-tp)\geqslant p-t$.
\end{proof}

Now let $p=p(\sqrt{n})$ and $t=p-\lfloor n/p\rfloor$. Thus
$n\geqslant p^2-tp$. Recall that $\kappa_{_{\leqslant 1}}(n,n)$
is an increasing function for $n$. By Claim \ref{claim},
$$
\kappa_{_{\leqslant 1}}(n,n)\geqslant \kappa_{\leqslant
1}(p^{2}-tp,p^{2}-tp)\geqslant
p-t=\left\lfloor\frac{n}{p}\right\rfloor.
$$
\end{proof}

\begin{definition}
A {\em projective plane} consists of a set of points, a set of
lines, and a relation between points and lines called incidence,
having the following properties:\\
$(1)$ Given any two distinct points, there is exactly one line incident with both of them;\\
$(2)$ Given any two distinct lines, there is exactly one point incident with both of them;\\
$(3)$ There are four points such that no line is incident with more
than two of them.
\end{definition}

It is not difficult to see that for every projective plane $\mathcal{P}$,
there exists an integer $q$ such that each point is incident with $q+1$ lines and each line is incident with $q+1$ points.
Such an integer $q$ is the \emph{order} of $\mathcal{P}$.
One can check that a projective plane of order $q$ has $q^2+q+1$ points and $q^2+q+1$ lines.
The following well-known result on the existence of finite projective planes will be used.

\begin{lemma}[see \cite{Coxeter}]
The projective plane of order $q$ exists if $q$ is prime power.
\end{lemma}

\begin{proof}[\bf{Proof of Theorem \ref{thm: sharpness of an upper bound}}]
Assume first that $n=q^{2}+q+1$. From Theorem \ref{thm: sqrt n}, we
have
$$
\kappa_{\leqslant 1}(q^{2}+q+1,q^{2}+q+1)\leqslant \sqrt{q^{2}+q+\frac{1}{4}}+\frac{1}{2}=q+1.
$$
Note that a projective plane of order $q$ is a $(q+1)$-uniform
$\{0,1\}$-intersecting $(q^{2}+q+1)$-family. Thus the
equality holds in the above inequality.

Now assume that $n\in[q^{2},q^{2}+q]$. From Theorem \ref{thm: sqrt
n}, we have $\kappa_{\leqslant 1}(n,n)\leqslant q$. Let
$X=[n]=\{1,\ldots,n\}$ and let $X'=\{1,\ldots,q^{2}+q,q^{2}+q+1\}$.
Since $\kappa_{_{\leqslant 1}}(q^{2}+q+1,q^{2}+q+1)=q+1$, there exists a
$(q+1)$-uniform family, say
$\mathcal{F}'=\{F'_1,\ldots,F'_{q^2+q+1}\}$ such that each two
member of $\mathcal{F}$ intersects on at most one element. Assume
without loss of generality that
$F'_{q^{2}+q+1}=\{q^{2}+1,\ldots,q^{2}+q+1\}$. For each $1\leqslant
i\leqslant n$, let $F_i$ be a set obtained from $F'_i$ by removing
its largest number. Since $|F'_{i}\cap F'_{q^{2}+q+1}|\leqslant 1$,
we have $F_{i}\subseteq \{1,\ldots,q^{2}\}\subseteq X$. Clearly,
$\mathcal{F}=\{F_1,\ldots,F_n\}$ is a $q$-uniform
$\{0,1\}$-intersecting family. So $\kappa_{_{\leqslant 1}}(n,n)=q$
for $n\in[q^{2},q^{2}+q]$.
\end{proof}

\section{Concluding remarks}\label{section: concluding}

We conclude this paper by proposing a conjecture on
estimating the maximum size of a member in a family among all uniform $L$-intersecting $m$-families of subsets of $[n]$ with $m=n$ and $L=\{0,1,\ldots,l\}$.

\begin{conjecture}
Let $l\geqslant 1$ be an integer.
Then
$\kappa_{_{\leqslant l}}(n,n)=(1+o(1))n^{\frac{l}{l+1}}$, i.e.,
$$
\lim_{n\rightarrow \infty}\frac{\kappa_{_{\leqslant l}}(n,n)}{n^{l/(l+1)}}=1.
$$
\end{conjecture}

\noindent
\textbf{Remark 1.}
By Theorem \ref{thm: sqrt n}, we have $\lim_{n\rightarrow \infty}\frac{\kappa_{_{\leqslant 1}}(n,n)}{\sqrt{n}}=1$,
i.e., the conjecture holds for $l=1$.
Taking advantage of Theorem \ref{thm: DEF theorem}, we can obtain that
$$
\lim_{n\rightarrow \infty}\frac{\kappa_{_{\leqslant l}}(n,n)}{n^{l/(l+1)}}\leqslant 1.
$$
So it suffices to show that
$\lim_{n\rightarrow \infty}\frac{\kappa_{_{\leqslant l}}(n,n)}{n^{l/(l+1)}}\geqslant 1$ for $l\geqslant 2$.

\end{spacing}

\begin{thebibliography}{0}



\bibitem{AHS1972}
H.L. Abbott, D. Hanson, N. Sauer,
Intersection theorems for systems of sets,
J. Combin. Theory Ser. A 12 (1972) 381-389.

\bibitem{ABS1991}
N. Alon, L. Babai, H. Suzuki,
Multilinear polynomials and Frankl-Ray-Chaudhuri-Wilson type intersection theorems,
J. Combin. Theory Ser. A 58 (1991) 165-180.

\bibitem{BF1980}
L. Babai, P. Frankl,
On set intersections,
J. Combin. Theory Ser. A 28 (1980) 103-105.

\bibitem{BFKS2001}
L. Babai, P. Frankl, S. Kutin, D. \v{S}tefankovi\v{c},
Set systems with restricted intersections modulo prime powers,
J. Combin. Theory Ser. A 95 (2001) 39-73.

\bibitem{Bose1949}
R.C. Bose,
A note on Fisher's inequality for balanced incomplete block designs,
Ann. Math. Stat. 20 (1949) 619-620.

\bibitem{CL2009}
W.Y.C. Chen, J. Liu,
Set systems with $\mathcal{L}$-intersections modulo a prime number,
J. Combin. Theory Ser. A 116 (2009) 120-131.

\bibitem{Coxeter}
H.S.M. Coxeter,
Projective Geometry,
2nd Edition, Springer-Verlag, 1987.

\bibitem{DEGKM1997}
W.A. Deuber, P. Erd\H{o}s, D.S. Gunderson, A.V. Kostochka, A.G. Meyer,
Intersection statements for systems of sets,
J. Combin. Theory Ser. A 79 (1997) 118-132.

\bibitem{DEF1978}
M. Deza, P. Erd\H{o}s, P. Frankl, Intersection properties of systems
of finite sets, Proceedings of London Mathematical Society 3 (2)
(1978) 369-384.

\bibitem{DF1983}
M. Deza, P. Frankl,
Erd\H{o}s-Ko-Rado theorem-22 years later,
SIAM J. Alg. Discrete Math. 4 (1983) 419-431.

\bibitem{Du2018}
P. Dusart,
Explicit estimates of some functions over primes,
The Ramanujan Journal 45 (1) (2018) 227-251.

\bibitem{EKR1961}
P. Erd\H{o}s, C. Ko, R. Rado,
Intersection theorems for systems of finite sets,
Quart. J. Math. Oxford Ser. 12 (2) (1961) 313-320.

\bibitem{Fisher1940}
R.A. Fisher,
An examination of the different possible solutions of a problem in incomplete blocks,
Annals Eugenics 10 (1940) 52-75.

\bibitem{Frankl1984}
P. Frankl,
Families of finite sets with three intersections,
Combinatorica 4 (1984) 141-148.

\bibitem{FF1983}
P. Frankl, Z. F\"{u}redi,
A new generalization of the Erd\H{o}s-Ko-Rado theorem,
Combinatorica 3 (1983) 341-349.

\bibitem{FT2016}
P. Frankl, N. Tokushige,
Invitation to intersection problems for finite sets,
J. Combin. Theory Ser. A 144 (2016) 157-211.

\bibitem{FW1981}
P. Frankl, R.M. Wilson,
Intersection theorems with geometric consequences,
Combinatorica 1 (1981) 357-368.

\bibitem{FS2004}
Z. F\"{u}redi, B. Sudakov,
Extremal set systems with restricted $k$-wise intersections,
J. Combin. Theory Ser. A 105 (2004) 143-159.

\bibitem{GS2002}
V. Grolmusz, B. Sudakov,
On $k$-wise set-intersections and $k$-wise hamming distances,
J. Combin. Theory Ser. A 99 (2002) 180-190.

\bibitem{Hegedus2015}
G. Heged\H{u}s,
A generalization of the Erd\H{o}s-Ko-Rado theorem,
arXiv:1512.05531v2.

\bibitem{HZ2017}
H. Huang, Y. Zhao, Degree versions of the Erd\H{o}s-Ko-Rado theorem
and Erd\H{o}s hypergraph matching conjecture, J. Combin. Theory Ser.
A 150 (2017) 233-247.

\bibitem{Isbell1959}
J.R. Isbell,
An inequality for incidence matrices,
Proc. Amer. Math. Soc. 10 (1959) 216-218.

\bibitem{LL2017}
J. Liu, X. Liu,
Set systems with positive intersection sizes,
Discrete Math. 340 (2017) 2333-2340.

\bibitem{LY2014}
J. Liu, W. Yang,
Set systems with restricted $k$-wise $\mathcal{L}$-intersections modulo a prime number,
European J. Combin. 36 (2014) 707-719.

\bibitem{LZLZ2016}
J. Liu, S. Zhang, S. Li, H. Zhang,
Set systems with $k$-wise $L$-intersections and codes with restricted Hamming distances,
European J. Combin. 58 (2016) 166-180.

\bibitem{Majumdar1953}
K.N. Majumdar,
On some theorems in combinatorics relating to incomplete block designs,
Annals Math. Statistics 24 (1953) 377-389.

\bibitem{Mubayi2007}
D. Mubayi,
An intersection theorem for four sets,
Advances in Math. 215 (2007) 601-615.

\bibitem{Mubayi2006}
D. Mubayi,
Erd\H{o}s-Ko-Rado for three sets,
J. Combin. Theory Ser. A 113 (2006) 547-550.

\bibitem{Pyber1986}
L. Pyber,
A new generalization of the Erd\H{o}s-Ko-Rado theorem,
J. Combin. Theory Ser. A 43 (1986) 85-90.

\bibitem{RCW1975}
D.K. Ray-Chaudhuri, R.M. Wilson,
On $t$-designs,
Osaka J. Math. 12 (1975) 737-744.

\bibitem{RT2006}
V. R\"{o}dl, E. Tengan,
A note on a conjecture by F\"{u}redi,
J. Combin. Theory Ser. A 113 (2006) 1214-1218.

\bibitem{Ryser1973}
H.J. Ryser, Intersection properties of finite sets, J. Combin.
Theory 14 (1973) 79-92.

\bibitem{Snevily1994}
H.S. Snevily, On generalizations of the deBruijin-Erd\H{o}s theorem,
J. Combin. Theory Ser. A 68 (1994) 232-238.

\bibitem{Snevily1999}
H.S. Snevily,
A generalization of Fisher's inequality,
J. Combin. Theory Ser. A 85 (1999) 120-125.

\bibitem{Snevily2003}
H.S. Snevily,
A sharp bound for the number of sets that pairwise intersect at $k$ positive values,
Combinatorica 23 (2003) 527-533.

\bibitem{Talbot2004}
J. Talbot,
The number of $k$-intersections of an intersecting family of $r$-sets,
J. Combin. Theory Ser. A 106 (2004) 277-286.

\end{thebibliography}
\end{document}